\documentclass[12pt,oneside,reqno]{amsart}

\usepackage{hyperref}

\usepackage[headsep=30pt,footskip=30pt,twoside=false,bottom=2.5cm]{geometry}

\usepackage[all]{xy}

\xymatrixcolsep{2cm}

\numberwithin{equation}{section}

\allowdisplaybreaks[1]

\usepackage{etoolbox}

\makeatletter

\patchcmd{\thesubsection}{\arabic}{\arabic}{}{}

\patchcmd{\@seccntformat}{\@secnumfont}{%
  \@secnumfont\expandafter\protect\csname format#1\endcsname}{}{}

\patchcmd{\@startsection}{\@afterindenttrue}{\@afterindentfalse}{}{}

\patchcmd{\subsection}{-.5em}{.3\linespacing}{}{}

\makeatother

\theoremstyle{plain}

\newtheorem{theorem}{Theorem}[section]

\newtheorem{proposition}[theorem]{Proposition}

\newtheorem{lemma}[theorem]{Lemma}

\newtheorem{corollary}[theorem]{Corollary}

\theoremstyle{remark}

\newtheorem{remark}[theorem]{Remark}

\newcommand{\Hom}[3][]{\ensuremath{\mathrm{Hom}_{#1} (#2, #3)}}

\newcommand{\Ker}[1]{\ensuremath{\mathrm{Ker} (#1)}}

\newcommand{\SKer}[1]{\ensuremath{\mathcal{K}er (#1)}}

\newcommand{\Pic}[1]{\ensuremath{\mathrm{Pic} (#1)}}

\newcommand{\At}[1]{\ensuremath{\mathrm{At} (#1)}}

\newcommand{\ad}[1]{\ensuremath{\mathrm{ad}  (#1)}}

\newcommand{\ENd}[1]{\ensuremath{\mathrm{End}  (#1)}}

\newcommand{\Img}[1]{\ensuremath{\mathrm{Im} (#1)}}

\newcommand{\cat}[1]{\ensuremath{\mathcal{#1}}}

\newcommand{\END}[2][]{\ensuremath{\mathcal{E}\mathit{nd}_{#1} (#2)}}

\newcommand{\id}[1]{\ensuremath{\mathbf{1}_{#1}}}

\renewcommand{\dim}[2][]{\ensuremath{\mathrm{dim}_{#1}(#2)}}

\newcommand{\N}{\ensuremath{\mathbb{N}}}

\newcommand{\Z}{\ensuremath{\mathbb{Z}}}

\newcommand{\Q}{\ensuremath{\mathbb{Q}}}

\newcommand{\p}{\ensuremath{\mathbf{P}}}

\newcommand{\C}{\ensuremath{\mathbb{C}}}

\newcommand{\tr}[1]{\ensuremath{\mathrm{tr}(#1)}}

\newcommand{\struct}[1]{\ensuremath{\mathcal{O}_{#1}}}

\newcommand{\DifF}[3][]{%
  \ensuremath{\mathcal{D}\mathit{iff}^{#1}(#2, #3)}}

\newcommand{\coh}[3]{\ensuremath{\mathrm{H}^{#1}(#2,#3)}}

\renewcommand{\bar}[1]{\ensuremath{\overline{#1}}}

\begin{document}

\title[Moduli space of logarithmic connections]{Moduli space of logarithmic connections  singular over a finite subset of a compact Riemann surface}
\author{Anoop Singh}
\address{Harish-Chandra Research Institute, HBNI \\ Chhatnag Road \\ Jhusi \\
  Prayagraj 211~019 \\ India}
  \email{anoopsingh@hri.res.in}

\subjclass[2010]{14D20, 14C22, 14E05}
  \keywords{Logarithmic connection, Moduli space, Picard group}

\begin{abstract}
Let $S$ be a finite subset of a compact connected Riemann surface $X$ of genus $g \geq 2$. Let $\cat{M}_{lc}(n,d)$ denote the moduli
space of pairs $(E,D)$, where $E$ is a holomorphic vector 
bundle over $X$ and $D$ is a logarithmic connection on $E$ singular over $S$, with fixed residues in the centre of
$\mathfrak{gl}(n,\C)$, where $n$ and $d$ are mutually corpime. Let $L$ denote a fixed line bundle with a logarithmic connection $D_L$ singular over $S$. Let 
$\cat{M}'_{lc}(n,d)$ and $\cat{M}_{lc}(n,L)$ be the moduli spaces parametrising all pairs $(E,D)$ such that
underlying vector bundle $E$ is stable and $(\bigwedge^nE, \tilde{D}) \cong (L,D_L)$ respectively.
  Let $\cat{M}'_{lc}(n,L) \subset \cat{M}_{lc}(n,L)$ be the Zariski open dense subset such that the underlying vector bundle is stable. We show that there is a natural compactification of $\cat{M}'_{lc}(n,d)$ and $\cat{M}'_{lc}(n,L)$ and  compute their Picard groups.
 We also show that $\cat{M}'_{lc}(n,L)$ and hence $\cat{M}_{lc}(n,L)$ do not have any non-constant algebraic functions but they admit non-constant holomorphic functions. We also study the Picard group
 and algebraic functions on the moduli space of 
 logarithmic connections singular over $S$, with arbitrary residues.

\end{abstract}

\maketitle

\section{Introduction}
\label{Intro}
Let $X$ be a compact Riemann surface of genus$(X) = g \geq 2$. Fix a 
finite subset $S = \{x_1, \ldots, x_m \} $ of $X$  such 
that $x_i \neq  x_j$ for all $i \neq j$. Let $E$ be a 
holomorphic vector 
 bundle over $X$ of rank $n \geq 1$ and degree $d$, where 
 $n$ and $d$ are mutually coprime.
 For each $j = 1, \ldots, m$, fix
 $\lambda_j \in \Q$ such that $n \lambda_j \in \Z$. We
 consider the pair $(E,D)$, where $D$ is a logarithmic 
 connection in $E$ singular over $S$ with residues 
 $Res(D,x_j) = \lambda_j \id{E(x_j)}$. For the 
 construction of the moduli space of logarithmic 
 connections, see \cite{N}, \cite{S1}. 
 
We have put the condition on the residue and on $\lambda_j \in \Q$,
that is, $n \lambda_j \in \Z$, to ensure the following 
 
 \begin{enumerate}
 \item The moduli space $\cat{M}_{lc}(n,d) \neq \emptyset$
 [see \cite{B}, Proposition $1.2$].
 \item Under the monodromy representation 
 $\rho: \pi_1(X_0,x_0) \to \text{SL}(n,\C)$, the image of the
path homotopy class of the loop $\gamma_j$ around $x_j \in S$
based at $x_0$, is a diagonal matrix with entries 
$\exp(-2 \pi \sqrt{-1} \lambda_j )$ [see section \eqref{Mod-log-conn}].
 \end{enumerate}

 In \cite{BR}, the moduli space 
 of rank $n$ logarithmic connections singular exactly 
 over one 
 point has been considered and several properties, like  
 algebraic functions, compactification  and computation 
 of Picard group have been studied.  Also, the
 moduli space of rank one logarithmic connections singular over finitely many points with fixed residues has been considered in \cite{Se}, and it is proved that it has a natural symplectic structure and there are no non-constant algebraic functions on it.

 In the present article, 
 our aim is to study the algebraic functions and Picard 
 group for the moduli space of rank $n$ logarithmic connections 
 singular over $S$ with fixed residues and hence, we end 
 up  generalising several results in \cite{BR}.

 Let $\cat{U}(n,d)$ denote the moduli space of all stable 
 vector bundles of rank $n$ and degree $d$ over $X$ and 
    $\cat{M}_{lc}(n,d)$  denote of the moduli space logarithmic connection 
 $(E,D)$ 
 singular over $S$ with fixed residues $Res(D,x_j) = 
 \lambda_j \id{E(x_j)}$ for all $j= 1, \ldots, m$, and 
 $\cat{M}'_{lc}(n,d) \subset \cat{M}_{lc}(n,d)$ be the 
 moduli space of logarithmic connection whose underlying 
 vector bundle is stable. We show that there is a natural
 compactification of the moduli space 
 $\cat{M}'_{lc}(n,d)$. More precisely, we prove the following (see subsection \ref{proof_thm_1.1} for the proof.)
 \begin{theorem}
 \label{thm:1.1}
 There exists an algebraic vector bundle $\pi: \Xi \to \cat{U}(n,d)$ such that 
$\cat{M}'_{lc}(n,d)$ is embedded in $\p (\Xi)$ with $\p (\Xi) \setminus \cat{M}'_{lc}(n,d)$ as the
hyperplane at infinity.
 \end{theorem}

 Let $\cat{U}(n,d)$ denote the moduli space of all stable 
 vector bundles over $X$. Then $\cat{U}(n,d)$ 
is an irreducible smooth complex projective variety of 
dimension $n^2(g-1)+1$ [see \cite{R}].
 
Now, we have the natural homomorphism 
\begin{equation}
\label{eq:0}
p: \cat{M}'_{lc}(n,d) \to \cat{U}(n,d)
\end{equation}
sending $(E,D)$ to $E$. 
The morphism $p$ induces a homomorphism 
\begin{equation}
\label{eq:1.1}
p^*: \Pic{\cat{U}(n,d)} \to  \Pic{\cat{M}'_{lc}(n,d)}
\end{equation}
of Picard groups,  
that sends an algebraic line bundle $\xi$ over $\cat{U}(n,d)$ to an algebraic line bundle $p^* \xi$ over $\cat{M}'_{lc}(n,d)$.
Here, $\Pic{Y}$ consists of algebraic line bundles over $Y$,
where $Y$ is an algebraic variety over $\C$.  we show the following (see subsection \ref{proof_thm_1.2} for the proof.)
 \begin{theorem}
 \label{thm:1.2}
 The homomorphism $p^*: \Pic{\cat{U}(n,d)} \to \Pic{\cat{M}'_{lc}(n,d)}$ is an isomorphism of groups.
 \end{theorem}
 
 Now, fix a holomorphic line bundle $L$ over $X$ of degree $d$, and fix a logarithmic connection $D_{L}$ on $L$ singular over $S$ with residues $Res(D_L, x_j) = n \lambda_j$ for all $j = 1, \ldots, m$. Let $\cat{M}_{lc}(n,L) \subset \cat{M}_{lc}(n,d)$ be the moduli space 
 parametrising isomorphism class of pairs $(E,D)$ such that $(\bigwedge^n E, \tilde{D}) \cong (L,D_L)$, where $\tilde{D}$ is the logarithmic connection on 
 $\bigwedge^n E$ induced by $D$. Let $\cat{M'}_{lc}(n,L) = \cat{M}_{lc}(n,L) \cap \cat{M}'_{lc}(n,d)$ and $\cat{U}_{L}(n,d) \subset \cat{U}(n,d)$ be the moduli space of stable vector bundles with $\bigwedge^n E \cong L$. 
 Similarly, we have a natural morphism 
\begin{equation}
\label{eq:1.25}
p_0: \cat{M}'_{lc}(n,L) \to \cat{U}_{L}(n,d)
\end{equation} 
 of varieties, that induces a homomorphism of Picard groups, and we have 
 \begin{proposition}
 \label{prop:1.2}
 The homomorphism $p_0^*: Pic(\cat{U}_{L}(n,d)) \to Pic(\cat{M}'_{lc}(n,L))$ defined by $\xi \mapsto p_0^* \xi$ is an isomorphism of groups.
 \end{proposition}
 
 From [\cite{R}, Proposition 3.4], we have $\Pic{\cat{U}_{L}(n,d)} \cong \Z$. Let $\Theta$ be the 
 ample generator of the group $\Pic{\cat{U}_{L}(n,d)}$.
 Then we have the Atiyah exact sequence (see \cite{A}) associated  to
 the line bundle $\Theta$ over $\cat{U}_{L}(n,d)$,
 \begin{equation}
 \label{eq:1.2}
0 \to \struct{\cat{U}_{L}(n,d)} \xrightarrow{\imath} 
\At{\Theta} \xrightarrow{\sigma}   T\cat{U}_{L}(n,d) 
\to 0,
\end{equation}
where $\At{\Theta} $ is called
\textbf{Atiyah algebra} of the holomorphic line bundle $\Theta$.
Let $\cat{C}(\Theta) \subset \At{\Theta}^*$  be the fibre
bundle over $\cat{U}_L(n,d)$ such that for every $U 
\subset \cat{U}_L(n,d)$ a holomorphic section of $\cat{C}
(\Theta)|_{U}$ gives a holomorphic splitting of $\eqref{eq:1.2}$. 
Then the two 
$\Omega^1_{\cat{U}_{L}(n,d)}$-\emph{torsors} $\cat{C}(\Theta)$ and $\cat{M}'_{lc}(n,L)$  on 
${\cat{U}_{L}(n,d)}$ are isomorphic.
Finally, we show that there is no non-constant algebraic 
function on $\cat{M}'_{lc}(n,L)$, by showing the following (see section \eqref{proof_thm_1.3} for the proof)
\begin{theorem}
\label{thm:1.3}
Assume that $ \text{genus}(X) \geq  3$. Then
\begin{equation}
\label{eq:1.3}
\coh{0}{\cat{C}(\Theta)}{\struct{\cat{C}(\Theta)}} =\C.
\end{equation}
\end{theorem}
Since $\cat{M}'_{lc}(n,L) \subset \cat{M}_{lc}(n,L)$ is an open dense subset,  $\cat{M}_{lc}(n,L)$ does 
not have any non-constant algebraic function.

In the last section\eqref{Mod-log-conn-arb}, we consider 
the moduli spaces $\cat{N}_{lc}(n,d)$ and $\cat{N}'_{lc}
(n,d)$ of logarithmic connections with arbitrary 
residues and show the similar results as Theorem 
\ref{thm:1.1} and Theorem \ref{thm:1.2}, see Theorem
\ref{thm:4} and Theorem \ref{thm:5}.
Consider the moduli spaces $\cat{N}_{lc}(n,L)$
and $\cat{N}'_{lc}(n,L)$ of logarithmic connections 
singular over $S$ with arbitrary residues in the centre 
of $\mathfrak{gl}(n,\C)$, and  let 
\begin{equation*}
\label{eq:1.4}
V = \{(\alpha_1, \ldots, \alpha_m) \in \C^m|~~ n \alpha_j \in \Z~~ \text{and}~~ d+ n \sum_{j=1}^{m} \alpha_j = 0 \}
\end{equation*}

Define a map 
\begin{equation}
\label{eq:1.5}
\Phi: \cat{N}'_{lc}(n,L) \to V
\end{equation}
by $(E,D) \mapsto (\tr{Res(D,x_1)}/n, \ldots, \tr{Res(D,x_m)}/n)$.
Then we show the following [see section \eqref{Mod-log-conn-arb}  for the proof].
\begin{theorem}
\label{thm:1.4}
Any algebraic function on $\cat{N}'_{lc}(n,L)$ factors through the surjective map $\Phi: \cat{N}'_{lc}(n,L) \to V$ as defined in \eqref{eq:1.5}.
\end{theorem}

We also show the similar result for $\cat{N}_{lc}(n,L)$,
see Theorem \ref{thm:7}.

\section{Preliminaries}
\label{Pre}                                                                                                                                                                                                                                                                                                                                                                                                                                                                                                                                                                        

We denote by  $S = x_1+ \cdots + x_m$ the reduced effective divisor 
on $X$ associated to the finite set $S$. Let $\Omega^1_X(\log S)$ denote the sheaf of 
logarithmic differential $1$-forms along $S$,  see \cite{S}. The notion of logarithmic connection was introduced by P. Deligne 
in \cite{D}. We recall the definition 
of logarithmic connection in a holomorphic vector bundle $E$ over $X$ singular 
over $S$ and residue of  the logarithmic connection on the points of $S$.
 
Let $E$ be a holomorphic vector bundle on $X$ of rank 
$n \geq 1$. We will denote the fibre of $E$ over any point $x 
\in X$ by $E(x)$.

A logarithmic connection  on $E$ singular over $S$ is a $
\C$-linear map 
\begin{equation}
\label{eq:3}
D : E \to E \otimes \Omega^1_X(\log S)  = E \otimes
\Omega^1_X \otimes \struct{X}(S)  
\end{equation}
which satisfies the Leibniz identity
\begin{equation}
\label{eq:4}
D(f s)= f D(s) + df \otimes s,
\end{equation}
where $f$ is a local section of \struct{X} and $s$ is a 
local section of $E$.

Let $D$ be a logarithmic connection in $E$ singular over 
$S$. For any $x_\beta \in S$, the fiber $\Omega^1_X 
\otimes \struct{X}(S)(x_\beta)$ is canonically identified 
with $\C$ by sending a meromorphic form to its residue at 
$x_\beta$. 

Let $v \in E(x_\beta)$ be any vector in the fiber of $E$
over $x_\beta$. Let $U$ be an open set around $x_\beta$ and
$s: U \to E$ be a holomorphic section of $E$ over $U$ 
such that $s(x_\beta) = v$. Consider the following 
composition 
\begin{equation}
\label{eq:5}
\Gamma(U,E) \to \Gamma(U, E \otimes\Omega^1_X \otimes 
\struct{X}(S)) \to E \otimes\Omega^1_X \otimes \struct{X}
(S)(x_\beta) = E(x_\beta),
\end{equation}
where the equality is given because of the identification
$\Omega^1_X \otimes \struct{X}(S)(x_\beta) = \C$.

Let $t$ be a local coordinate at $x_\beta$ on $U$ such that 
$t(x_\beta) = 0$, that is, the coordinate system $(U,t)$ is centered at 
$x_\beta$ and suppose that  $\sigma \in 
\Gamma(U,E)$ such that $\sigma(x_\beta) = 0$. Then $\sigma = t  \sigma'$ 
for some $\sigma' \in \Gamma(U,E)$. Now,
\begin{align*}
D(\sigma) = D(t \sigma')& = t D(\sigma') + dt \otimes \sigma' \\
& = t D(\sigma') + t (\frac{dt}{t} \otimes \sigma'),
\end{align*}
and $D(\sigma)(x_\beta) = 0$. Thus, we have a well defined 
endomorphism, denoted by 
\begin{equation}
\label{eq:6}
Res(D,x_\beta) \in \ENd{E}(x_\beta) = \ENd{E(x_\beta)}
\end{equation}
that sends $v$ to $D(s)(x_\beta)$.  This endomorphism 
$Res(D,x_\beta)$ is called the \textbf{residue} of the logarithmic 
connection $D$ at the point $x_\beta \in S$ (see \cite{D} for the details). 

If $D$ is a logarithmic connection in $E$ singular over $S$ and 
$\theta \in \coh{0}{X}{\Omega^1_X \otimes \ENd{E}}$, then $D + \theta$
is also a logarithmic connection in $E$, singular over $S$. Also,
we have 
\begin{equation*}
\label{eq:11}
Res(D, x_\beta) = Res(D + \theta, x_\beta),
\end{equation*}
  for every $ x_\beta \in S$.
  
  Conversely, if $D$ and $D'$ are two logarithmic connections on $E$
  singular over $S$ with 
  \begin{equation}
  \label{eq:12}
  Res(D,x_\beta) = Res(D', x_\beta),
  \end{equation}

then $D' = D + \theta$, where $ \theta \in \coh{0}{X}{\Omega^1_X \otimes \ENd{E}}$.

Thus, the space of all logarithmic connections $D'$ on a given 
holomorphic vector bundle $E$ singular over $S$, and satisfying 
\eqref{eq:12} with $D$ fixed, is an affine space for 
$\coh{0}{X}{\Omega^1_X \otimes \ENd{E}}$.

For each $i = 1, \ldots, m$, fix $\lambda_i \in \Q$
such that $n \lambda_i \in \Z$, where $n$ is the rank of 
the vector bundle $E$.  
By a pair $(E,D)$ over $X$, we mean that
\begin{enumerate}
\item $E$ is a holomorphic 
vector bundle of degree $d$ and rank $n$ over $X$.
\item $D$ is 
a logarithmic 
connection in $E$ singular over $S$ with residues 
$Res(D,x_i) = \lambda_i \id{E(x_i)}$  for all $i= 1, 
\ldots,m$.
\end{enumerate}
Then from \cite{O}, Theorem 3, we have 
\begin{equation}
\label{eq:15}
d + n \sum_{j=1}^{m} \lambda_i = 0
\end{equation}

\begin{lemma}
\label{lem:1}
Let $(E,D)$ be a logarithmic connection on $X$. 
Suppose that $F$ is a holomorphic subbundle of $E$ such 
that the restriction $D^\prime = D|_F$ of $D$ to $F$ is a 
logarithmic connection in $F$ singular over $S$. Then 
$Res(D^\prime,x_j) = \lambda_j \id{F(x_j)}$ for all $j = 1, \ldots, m$.
\end{lemma}
\begin{proof}
Follows from the  definition of residues.
\end{proof}

A logarithmic connection $D$ in a holomorphic vector bundle $E$ 
is called \textbf{irreducible}  if for any holomorphic  subbundle $F$
of $E$ with 
$D(F) \subset \Omega^1_X(\log S) \otimes F $, then either $F = E$ or 
$F = 0$.

\begin{proposition}
\label{pro:2} 
Let $(E,D)$ be a logarithmic connection on $X$.
Suppose that $n$ and $d$ are mutually coprime. Then  $D$ is irreducible.
\end{proposition}
\begin{proof}
Let $0 \neq F$ be a holomorphic subbundle of $E$ of rank $r$ invariant under 
$D$, that is,  $D(F) \subset F \otimes \Omega^1_X(\log S)$. Set $D^\prime = D|_F$.
Then from Lemma \ref{lem:1}, $Res(D^\prime,x_i) = \lambda_i \id{F(x_i)}$, and from
\cite{O} Theorem 3, we have 
\begin{equation}
\label{eq:16}
\text{degree}(F)+ r \sum_{i=1}^m \lambda_i = 0.
\end{equation}
From \eqref{eq:15} and \eqref{eq:16}, we get that $\mu(F) = \mu(E)$. Since $F$ is a subbundle of $E$, if rank of $F$ is 
less than rank of $E$, we get that $n|d$, which is a contradiction. Thus $F = E$.
\end{proof}

\section{The moduli space of logarithmic connections with fixed residues}
\label{Mod-log-conn}
 We say two pairs $(E,D)$ and $(E',D')$ of rank $n$ and degree 
$d$  are isomorphic if there 
exists an isomorphism $\Phi: E \to E'$ such that the following diagram  
\begin{equation}
\label{eq:18}
\xymatrix{
E \ar[d]^{\Phi} \ar[r]^D & E \otimes \Omega^1_X(\log S) \ar[d]^{\Phi \otimes \id{\Omega^1_X(\log S)}} \\
E' \ar[r]^{D'} & E' \otimes \Omega^1_X(\log S)\\
}
\end{equation}
commutes.

Let $\cat{M}_{lc}(n,d)$ denote the moduli space which 
parametrizes the isomorphic class of  pairs $(E,D)$.
Then $\cat{M}_{lc}(n,d)$ is a separated quasi-projective 
scheme over $\C$ [see \cite{N}, Theorem 3.5].

Henceforth, we will assume following conditions
\begin{enumerate}
\item  $d$ and $n$ are  mutually coprime.
\item   for each $i = 1, \ldots, m$, $\lambda_i \in \Q$
such that $n \lambda_i \in \Z$.
\item $d,n, \lambda_1,..., \lambda_m$ satisfies following
relation
\begin{equation}
\label{eq:19}
d + n \sum_{i=1}^{m} \lambda_i = 0.
\end{equation}
\end{enumerate}

Under the above condtions, from the Proposition \ref{pro:2}, every logarithmic connection $(E,D)$
in $\cat{M}_{lc}(n,d)$ is irreducible. 
Since the singular points of  $\cat{M}_{lc}(n,d)$ corresponds 
to reducible logarithmic connections (see second paragraph on p.n. 
$790$ of \cite{BR}),
 the moduli
space $\cat{M}_{lc}(n,d)$ is smooth. 

The similar technique as in  \cite{S2}, Theorem $11.1$, can be 
used to show that $\cat{M}_{lc}(n,d)$ is an irreducible variety.
Thus,  altogether $\cat{M}_{lc}(n,d)$ is an irreducible smooth 
 quasi-projective variety over $\C$.

 From [\cite{M} Theorem 2.8(A)],  $\cat{M'}_{lc}(n,d) $ is a Zariski open subset of $\cat{M}_{lc}(n,d)$.  Since $\cat{M}_{lc}(n,d)$ irreducible, 
 $\cat{M}'_{lc}(n,d)$ is dense.

Consider $\cat{M}_{lc}(n,L)$ as defined above. Then it  is a closed subvariety of 
$\cat{M}_{lc}(n,d)$. Moreover,
 $\cat{M'}_{lc}(n,L) = \cat{M}_{lc}(n,L) \cap \cat{M}'_{lc}(n,d)$ is a  Zariski open dense subset of $\cat{M}_{lc}(n,L)$.

In particular, if we take  $L_0 = \otimes_{i = 1}^{m}
\struct{X}(-n \lambda_i x_i) $ and $D_{L_0}$ the 
logarithmic connection defined by the de Rham 
differential, then  $D_{L_0}$ is singular 
over $S$ with residues 
$Res(D_{L_0},x_i) = n \lambda_i $ for all $i = 1, 
\ldots,m$. For this pair $(L_0,D_{L_0})$ we denote 
the moduli spaces $\cat{M}_{lc}(n,L)$ and $\cat{M'}_{lc}
(n,L)$ by $\cat{M}_{lc}(n,L_0)$ and 
$\cat{M'}_{lc}(n,L_0)$ respectively.

Let $X_0 = X 
\setminus S$ and $x_0 \in X_0$.
Let $U_j$ be a simply connected open set in $X_0 \cup 
\{x_j\}$ containing $x_0$ and $x_j$.  Then $\pi_1(U_j 
\setminus \{x_j\}, x_0) \cong \Z$, where $1$ corresponds 
to the anticlockwise loop around $x_j$. We have a natural
 group homomorphism 
\begin{equation*}
\label{eq:20.9}
h_j: \pi_1(U_j\setminus \{x_j\}, x_0) \to \pi_1(X_0, 
x_0).
\end{equation*}
for all $j = 1, \ldots,m $.
Suppose that $h_j(1) = \gamma_j$ for all $j = 1, \ldots,m
$. Then $\pi_1(X_0,x_0)$
admits a presentation with $2g+m$ generators $a_1,b_1, 
\ldots, a_g,b_g, \gamma_1, \ldots, \gamma_m$ with 
relation $\Pi_{i=1}^{g} [a_i, b_i] \Pi_{j=1}^{m} \gamma_j 
= 1$.  

Let $(E,D) \in \cat{M}_{lc}(n,L_0)$. Then $D$ determines
a holomorphic (flat) connection on the holomorphic vector
bundle $E|_{X_0}$ restricted to $X_0$. Since $Res(D,x_j)
= \lambda_j \id{E(x_j)}$, for $j = 1, \ldots, m$, the 
image of $\gamma_j$ under the monodromy representation
is the $n \times n$ diagonal matrix with $\exp(-2 \pi 
\sqrt{-1} \lambda_j )$ (see \cite{D}, p.79, Proposition 3.11). 
Let $\cat{R}_g \subset \Hom{\pi_1(X_0,x_0)}{\text{SL}(n,
\C)}$denote the space of those representations $\rho: 
\pi_1(X_0,x_0) \to \text{SL}(n,\C)$ such that $
\rho(\gamma_j) = \exp(- 2 \pi \sqrt{-1} \lambda_j) 
\text{I}_{n \times n}$ for all $j = 1, \ldots, m$, where $\text{I}_{n \times n}$ 
denotes the $n \times n$ identity matrix.  Since the 
logarithmic connection $D$ is irreducible, any representation in $\cat{R}_g$  is irreducible. 
Consider the action of $\text{SL}(n, \C)$ on $\cat{R}_g$
by conjugation, that is, for any $T \in \text{SL}(n, \C)$
and  $\rho \in \cat{R}_g$ the action is defined by
$\rho . T = T^{-1} \rho T$. Let $\cat{B}_g = \cat{R}_g / 
\text{SL}(n, \C)$ be the quotient space for the 
conjugation action. The algebraic structure of $\cat{R}_g
$ induces an algebraic structure on $\cat{B}_g$. In 
literature, $\cat{B}_g$ is known as \textbf{Betti moduli 
space} (for instance see \cite{S1}, \cite{S2}) and it is an
irreducible smooth quasi-projective variety over $\C$.
Thus, we have a holomorphic  map 
\begin{equation}
\label{eq:21}
\Phi: \cat{M}_{lc}(n,L_0) \to \cat{B}_g
\end{equation}
sending $(E,D)$ to the equivalence class of its monodromy 
representaion under the conjugation action of $\text{SL}(n,\C)$.

For the inverse map of $\Phi$,  let $\rho \in \cat{B}_g$. 
Let $(E_\rho,\nabla_\rho)$ be the flat holomorphic vector 
bundle over $X_0$ associated to $\rho$. Then 
$E_\rho$ over $X_0$ extends to a holomorphic vector 
bundle $\bar{E_\rho}$ over $X$, and  the connection $
\nabla_\rho$ on $E_\rho$ extends to a connection 
$\bar{\nabla_\rho}$ such that $(\bar{E_
 \rho},\bar{\nabla_\rho}) \in \cat{M}_{lc}(n,L_0)$[See 
 \cite{BM}, p.159, Theorem 4.4]. Thus, $\Phi$ is a biholomorphism.

\section{The Picard group of moduli space of logarithmic 
connections}
\label{Pic}

Let 
 \begin{equation*}
 \label{eq:20}
 p: \cat{M}'_{lc}{(n,d)} \to \cat{U}(n,d)
 \end{equation*}
 be the forgetful map which forgets its 
 logarithmic structure as defined in \eqref{eq:0}.
 
 Let $E \in \cat{U}(n,d)$. Then $E$ is indecomposable. 
Since $d,n$ satisfy equation $\eqref{eq:19}$, from 
\cite{B},
Proposition $1.2$, $E$ admits a logarithmic connection $D$ singular 
over $S$, with residues $Res(D,x_j) = \lambda_j
\id{E(x_j)}$ for all $j = 1, \ldots, m$. 

Thus, the pair $(E,D)$ is in the moduli space $
\cat{M}'_{lc}(n,d)$, and hence $p$ is surjective.

\subsection{Torsors}
We recall the definition of torsors and will show that 
the map $p: \cat{M}'_{lc}{(n,d)} \to \cat{U}(n,d)$ is a $
\Omega^1_{\cat{U}(n,d)}$-\emph{torsor} on $\cat{U}(n,d)$, 
where $\Omega^1_{\cat{U}(n,d)}$ denotes the holomorphic
cotangent bundle over $\cat{U}(n,d)$.

Let $M$ be a connected complex manifold. Let $\pi: 
\cat{V} \to M$, be a 
holomorphic vector bundle.

A $\cat{V}$-\emph{torsor} on $M$ is a holomorphic fiber 
bundle $p: Z \to M$,
and holomorphic map from the fiber product
\begin{equation*}
\label{eq:20.1}
\varphi: Z \times_M \cat{V} \to Z
\end{equation*} 

such that
\begin{enumerate}
\item $p \circ \varphi = p \circ p_Z$, where $p_Z$ is the 
natural projection
of   $Z \times_M \cat{V} $ to $Z$,
\item the map $Z \times_M \cat{V} \to Z \times_M  Z$ 
defined by $p_Z \times 
\varphi$ is an isomorphism,
\item $\varphi(\varphi(z,v),w) = \varphi(z, v+w)$.
\end{enumerate}

\begin{proposition}
\label{prop:4.5}
The isomorphic classes of $\cat{V}$-torsors over $M$ are 
parametrized by $\coh{1}{M}{\cat{V}}$.
\end{proposition}

\begin{proposition}
\label{prop:5} Let $p:\cat{M}'_{lc}{(n,d)} \to \cat{U}
(n,d)$ be the map as defined in 
\eqref{eq:0}. Then $\cat{M}'_{lc}{(n,d)}$ is a $
\Omega^1_{\cat{U}(n,d)}$-\emph{torsor} on 
${\cat{U}(n,d)}$.
\end{proposition}
\begin{proof}
Let $E \in {\cat{U}(n,d)}$. Then $p^{-1}(E) \subset 
\cat{M}'_{lc}(n,d)$ is an affine space 
for $\coh{0}{X}{\Omega^1_X \otimes \ENd{E}}$ and the 
fiber of the cotangent bundle $ \pi:
\Omega^1_{\cat{U}(n,d)}\to \cat{U}(n,d)$ at $E$ is 
isomorphic to $\coh{0}{X}{\Omega^1_X \otimes \ENd{E}}$, 
that is, 
$\Omega^1_{\cat{U}(n,d),E} \cong \coh{0}{X}{\Omega^1_X 
\otimes \ENd{E}}$. There is a natural
action of $\Omega^1_{\cat{U}(n,d),E}$ on $p^{-1}(E)$, 
that is,
\begin{equation*}
\label{eq:20.2}
\Omega^1_{\cat{U}(n,d),E} \times p^{-1}(E) \to p^{-1}(E)
\end{equation*}
sending $(\omega, D)$ to $\omega + D$.
This action on the fibre is faithful and transitive.
This action will induce a holomorphic map on the fibre 
product
\begin{equation}
\label{eq:20.3}
\varphi: \Omega^1_{\cat{U}(n,d)} \times_{\cat{U}(n,d)} \cat{M}'_{lc}{(n,d)} \to \cat{M}'_{lc}{(n,d)},
\end{equation}
which satisfies the above conditions in the definition of 
the torsor.

\end{proof}

\subsection{The Picard group of  moduli space of logarithmic connection}
\label{Picard}

\begin{remark}[]
Note that $p: \cat{M}'_{lc}(n,d) \to \cat{U}(n,d)$ as defined in \eqref{eq:0} is a fibre bundle (not a vector bundle)  with 
fibre $p^{-1}(E)$ which is an affine space modelled over $\coh{0}{X}{\Omega^1_X \otimes \ENd{E}}$. Moreover, from 
Proposition \ref{prop:5}, $\cat{M}'_{lc}{(n,d)}$ is a $\Omega^1_{\cat{U}(n,d)}$-\emph{torsor} on 
${\cat{U}(n,d)}$.  We know that the dual of an affine space (modelled over a 
vector space over $\C$) is a vector space over $\C$, in the same spirit, 
 the dual of a torsor
is a vector bundle. We use this fact to construct an algebraic vector bundle over 
${\cat{U}(n,d)}$. For another construction of this algebraic 
vector bundle over ${\cat{U}(n,d)}$, see the Remark $3.2$ and the third paragraph 
on the p.n.$792$ in \cite{BR}.
\end{remark}

\begin{proof}[\bf Proof of Theorem  \ref{thm:1.1}]
\label{proof_thm_1.1}
For
any $E \in \cat{U}(n,d)$, the fiber $p^{-1}(E)$ is an affine space modelled on
 $\coh{0}{X}{\Omega^1_X \otimes \ENd{E}}$.  The dual $$p^{-1}(E)^{\vee} = \{\varphi: p^{-1}
 (E) \to \C  ~\vert~ \varphi~ \mbox{is an affine linear map} \}$$  is a vector space over $\C$. 
 
 Let $\pi: \Xi \to \cat{U}(n,d)$ be the algebraic vector bundle such that for every
 Zariski open subset $U$ of $\cat{U}(n,d)$, a section of $\Xi$ over $U$ is an  
 algebraic function $f: p^{-1}(U) \to \C$ whose restriction to each fiber 
 $p^{-1}(E)$,  is an element of $p^{-1}(E)^{\vee}$. Thus, a fiber
 $\Xi(E) = \pi^{-1}(E)$ of
 $\Xi$ at $E \in \cat{U}(n,d)$ is $p^{-1}(E)^{\vee}$. Let $(E,D) \in \cat{M}'_{lc}(n,d)$,
 and define a map $\Phi_{(E,D)}: p^{-1}(E)^{\vee} \to \C$, by $\Phi_{(E,D)}(\varphi) = \varphi[(E,D)]$, which is nothing but the evaluation map. Now,
 the kernel $\Ker{\Phi_{(E,D)}}$ defines a hyperplane in $p^{-1}(E)^{\vee}$
 denoted by $H_{(E,D)}$. Let $\p (\Xi)$ be a projective bundle defined by hyperplanes in the fiber $p^{-1}(E)^{\vee}$, that is, we have $\tilde{\pi}: \p(\Xi) \to \cat{U}(n,d)$ induced from $\pi$. Define a map $\iota: \cat{M}'_{lc}(n,d)
  \to \p(\Xi)$ by sending  $(E,D)$ to the equivalence class of  $H_{(E,D)}$,
  which is clearly an open embedding. Set $Y = \p (\Xi) \setminus \cat{M}'_{lc}(n,d)$.
  Then $\tilde{\pi}^{-1}(E) \cap Y$ is a projective hyperplane in $\tilde{\pi}^{-1}(E)$ for every $E \in \cat{U}(n,d)$, and hence $Y$ is a hyperplane at infinity. This completes the proof.
  \end{proof}

The techniques used in the proof of the following theorem are
adopted from the proof of the  Theorem 3.1 in \cite{BR}.

\begin{proof}[\bf Proof of Theorem \ref{thm:1.2}]
\label{proof_thm_1.2}
First we show that $p^*$ in \eqref{eq:1.1} is injective.  Let $\xi \to 
\cat{U}(n,d)$ be an algebraic line bundle such that $p^* \xi$ is a trivial line bundle
over $\cat{M}'_{lc}(n,d)$. Giving a trivialization of $p^* \xi$ is equivalent to
giving a nowhere vanishing section of $p^* \xi$ over $\cat{M}'_{lc}(n,d)$. Fix 
$s \in \coh{0}{\cat{M}'_{lc}(n,d)}{p^* \xi}$ a nowhere vanishing section. Take any
point $E \in \cat{U}(n,d)$.  Then, 
\begin{equation*}
\label{eq:C}
s|_{p^{-1}(E)}: p^{-1}(E) \to \xi(E)
\end{equation*}
is a nowhere vanishing map. Notice that $p^{-1}(E) \cong \C^N$ and $\xi(E) \cong \C$, where $N = n^2(g-1)+1$. Now, any nowhere vanishing algebraic function on an affine space
$\C^N$ is a constant function, that is, $s|_{p^{-1}(E)}$ is a constant
function and hence corresponds to a non-zero vector $\alpha_{E} \in \xi(E)$.
Since $s$ is constant on each fiber of $p$,  the trivialization $s$ of $p^*\xi$ descends to a trivialization of the
line bundle $\xi$ over $\cat{U}(n,d)$, and hence giving a nowhere vanishing
section of $\xi$ over $\cat{U}(n,d)$. Thus, $\xi$ is a trivial line bundle
over $\cat{U}(n,d)$.  The surjectivity of $p^*$ follows from the Theorem
\ref{thm:1.1} and the fact that $ \Pic {\p (\Xi)} \cong \tilde{ \pi}^*\Pic{\cat{U}(n,d)}\oplus \Z \struct{\p (\Xi)}(1)$.

\end{proof}  

\section{Algebraic functions on the moduli space}
\label{fun}
Let $\Omega^1_{\cat{U}_{L}(n,d)}$ denote the holomorphic cotangent bundle on $\cat{U}_{L}(n,d)$.
Then, we have following proposition.

\begin{proposition}
\label{prop:7}
Let $p_0:\cat{M}'_{lc}{(n,L)} \to \cat{U}_{L}
(n,d)$ be the map as defined in 
\eqref{eq:1.25}. Then $\cat{M}'_{lc}{(n,L)}$ is a 
$\Omega^1_{\cat{U}_{L}(n,d)}$-\emph{torsor} on 
${\cat{U}_{L}(n,d)}$.
\end{proposition}  
 \begin{proof}
 First note that for any $E \in \cat{U}_{L}(n,d)$, the 
 holomorphic cotangent space $\Omega^1_{\cat{U}_{L}
 (n,d),E}$ at $E$ is isomorphic to $\coh{0}{X}{\Omega^1_X 
 \otimes \ad{E}}$, where $\ad{E} \subset \ENd{E}$ is the 
 subbundle consists of endomorphism of $E$ whose trace is 
 zero. Also, $p_0^{-1}(E)$ is an affine space modelled 
 over $\coh{0}{X}{\Omega^1_X \otimes \ad{E}}$. Thus, 
 there is a natural action of $\Omega^1_{\cat{U}_{L}
 (n,d),E}$  on $p_{0}^{-1}(E)$, that is,
\begin{equation*}
\label{eq:22.1}
\Omega^1_{\cat{U}_{L}(n,d),E} \times p_{0}^{-1}(E) \to 
p_0^{-1}(E)
\end{equation*}
sending $(\omega, D)$ to $\omega + D$, which is faithful
and transitive.
 \end{proof}
 
 \begin{proposition}
 \label{prop:7.1}
 There exists an algebraic vector bundle $\pi:\Xi' \to \cat{U}_L(n,d)$ such that 
$\cat{M}'_{lc}(n,L)$ is embedded in $\p (\Xi')$ with $\p (\Xi') \setminus \cat{M}'_{lc}(n,L)$ as the
hyperplane at infinity.
\end{proposition} 
\begin{proof} See the proof of the  Theorem \ref{thm:1.1}.
\end{proof}

 \begin{proposition}
 \label{prop:8}
 The homomorphism $p_0^*: Pic(\cat{U}_{L}(n,d)) \to Pic(\cat{M}'_{lc}(n,L))$ defined by $\xi \mapsto p_0^* \xi$ is an isomorphism of groups.
 \end{proposition}
\begin{proof} See the proof of the Theorem \ref{thm:1.2}.
\end{proof}

Now, from \cite{R}, Proposition 3.4, (ii), we have $Pic(\cat{U}_{L}(n,d)) \cong \Z$. Thus, in view of 
Proposition \ref{prop:8}, we have 
\begin{equation}
\label{eq:23}
Pic(\cat{M}'_{lc}(n,L)) \cong \Z.
\end{equation}  

Let $\Theta$ be the ample generator of the group $Pic(\cat{U}_{L}(n,d))$. 
We have the \emph{symbol exact sequence} for the holomorphic 
line bundle $\Theta$ given as follows, 
\begin{equation}
\label{eq:24}
0 \to \END[\struct{\cat{U}_{L}(n,d)}]{\Theta} 
\xrightarrow{\imath} \DifF[1]{\Theta}{\Theta} 
\xrightarrow{\sigma}  T\cat{U}_{L}(n,d) \otimes 
{\END[\struct{\cat{U}_{L}(n,d)}]{\Theta}} \to 0,
\end{equation}
where $\DifF[1]{\Theta}{\Theta}$ denotes the sheaf of 
first order holomorphic differential operator from $
\Theta$ to itself, and $T\cat{U}_{L}(n,d)$ is the 
holomorphic tangent bundle over $\cat{U}_{L}(n,d)$.
Since $\Theta$ is a holomorphic line bundle, the \emph{symbol 
exact sequence} \eqref{eq:24} becomes \eqref{eq:1.2} because in that case
  $ \At{\Theta} = \DifF[1]{\Theta}{\Theta}$, for more details see \cite{A}, 
\cite{BR1}.

Dualising the exact sequence \eqref{eq:1.2},
we get following exact sequence,

\begin{equation}
\label{eq:26}
0 \to \Omega^1_{\cat{U}_{L}(n,d)} 
\xrightarrow{\sigma^*} 
\At{\Theta}^* \xrightarrow{\imath^*}   \struct{\cat{U}
_{L}(n,d)} \to 0
\end{equation}

Consider $\struct{\cat{U}_{L}(n,d)}$ as trivial line 
bundle $\cat{U}_{L}(n,d) \times \C$. Let $s: \cat{U}
_{L}(n,d) \to \cat{U}_{L}(n,d) \times \C$ be a
holomorphic map defined by $E \mapsto (E,1)$. Then $s$ is
a holomorphic  section of the trivial line bundle $\cat{U}_{L}(n,d) \times \C$.

Let $S = \Img{s} \subset \cat{U}_{L}(n,d) \times \C$ be the image of $s$. 
Then $S \to \cat{U}_{L}(n,d)$ is a fibre bundle.
Consider the inverse image ${\imath^*}^{-1}S \subset 
\At{\Theta}^*$, and denote it by $\cat{C}(\Theta)$.
Then for every open subset $U \subset \cat{U}_{L}(n,d) 
$, a holomorphic section of $\cat{C}(\Theta)|_{U}$ over $U$ gives a holomorphic
splitting of  \eqref{eq:1.2}. 
 For instance, suppose $\gamma: U \to \cat{C}(\Theta)|
_{U}$ is a holomorphic section. Then $\gamma$ will be a 
holomorphic section of $\At{\Theta}^*|_{U}$  over $U$, 
because $ \cat{C}(\Theta) =
{\imath^*}^{-1}S \subset \At{\Theta}^*$. Since $ 
 \gamma \circ \imath =  \imath^*(\gamma) = \id{U}$, so we 
 get a holomorphic  splitting $\gamma$ of \eqref{eq:1.2}.
  Thus, $\Theta|_{U}$ admits 
 a holomorphic connection. 
 Conversely, given any holomorphic splitting of 
 \eqref{eq:1.2} over an open subset $U \subset \cat{U}_{L}(n,d)$, we get a holomorphic section of $\cat{C}(\Theta)|_{U}$.
 
 Let
 \begin{equation}
 \label{eq:26.1}
 \psi: \cat{C}(\Theta) \to \cat{U}_{L}(n,d)
 \end{equation}
 be the canonical projection.
Then using the short exact sequence \eqref{eq:26}, $\cat{C}(\Theta)$ is a 
$\Omega^1_{\cat{U}_{L}(n,d)}$-\emph{torsor} on 
${\cat{U}_{L}(n,d)}$

\begin{proposition}
\label{lem:3}
There is an isomorphism of algebraic varieties
\begin{equation}
\label{eq:27.1} f : \cat{C}(\Theta) \to \cat{M}'_{lc}{(n,L)}
\end{equation}
such that $p_0 \circ f = \psi$, where $p_0$ and $\psi$ 
are defined in \eqref{eq:1.25} and \eqref{eq:26.1} 
respectively.
\end{proposition} 
\begin{proof}
From the Proposition \ref{prop:4.5}, isomorphism class
of $\Omega^1_{\cat{U}_{L}(n,d)}$-torsors over $\cat{U}_{L}(n,d)$ is given by a cohomology class in $
\coh{1}{\cat{U}_L(n,d)}{\Omega^1_{\cat{U}_L(n,d)}}$.
 
Let $\alpha, \beta \in \coh{1}{\cat{U}_L(n,d)}{\Omega^1_{\cat{U}_L(n,d)}}$ be the cohomology class 
corresponding to $\cat{C}(\Theta)$ and $\cat{M}'_{lc}{(n,L)}$  respectively. Since the $\dim[\C]{\coh{1}{\cat{U}_L(n,d)}{\Omega^1_{\cat{U}_L(n,d)}}} = 1$, there
exists $c \in \C$ such that $\beta = c~\alpha$.
Thus, $\cat{C}(\Theta)$ and $\cat{M}'_{lc}{(n,L)}$ are
isomorphic as a fibre bundle over  $\cat{U}_L(n,d)$.
Now, to complete the proof, 
it is sufficient to show that $\alpha \neq 0$ and $\beta \neq 0$.  $\Theta$  being an ample line bundle, its
 first Chern class $c_1(\Theta) \neq 0$  and $\alpha = c_1(\Theta)$. From \cite{BR2}, Theorem 2.11, we conclude
 that $\beta \neq 0$.
 
\end{proof}

Let $\alpha_j \in \Q$, for $j = 1, \ldots, m$,  such 
that $n \alpha_j \in \Z$ and $d + n \sum_{j=1}^m 
\alpha_j = 0$.  Fix a holomorphic line bundle $L$ of
degree $d$, and fix a logarithmic connection $D'_L$ on
$L$ singular over $S$ with residues $Res(D'_L,x_j) = n
\alpha_j $ for $j = 1, \ldots, m$.

Let $\cat{V}_{lc}(n,L)$ denote the moduli 
space parametrising all pairs $(E,D)$ such that

\begin{enumerate}
\item $E$ is a holomorphic vector bundle of rank $n$ over $X$
with $\bigwedge^{n}E \cong L $.
\item $D$ is a logarithmic connection on $E$ singular over $S$
with $Res(D,x_i) = \alpha_i \id{E(x_i)}$, for every $i = 1, \ldots, m$.
\item the logarithmic connection on $\bigwedge^{n}E$ induced by
$D$ coincides with the given logarithmic connection $D'_{L}$ on $L$.
\end{enumerate}
Let $\cat{V}'_{lc}(n,L)$ denote the subset of 
$\cat{V}_{lc}(n,L)$ whose underlying vector bundle is
stable.

From Proposition \ref{lem:3}, we have 
\begin{corollary}
\label{cor:1.1}
There is an isomorphism between $\cat{M}'_{lc}(n,L)$
and $\cat{V}'_{lc}(n,L)$. 
\end{corollary}
\begin{proof}
From above Proposition $\ref{lem:3}$ both the 
varieties are isomorphic to $\cat{C}(\Theta)$.
\end{proof}

\begin{corollary}
\label{cor:1.2}
$\cat{M}_{lc}(n,L)$ and $\cat{V}_{lc}(n,L)$ are  
birationally equivalent.
\end{corollary}
\begin{proof}
Since $\cat{M}_{lc}(n,L)$ and $\cat{V}_{lc}(n,L)$ being
irreducible quasi-projective varieties over $\C$, and
$\cat{M}'_{lc}(n,L)$ and $\cat{V}'_{lc}(n,L)$ are
dense open subset of $\cat{M}_{lc}(n,L)$ and $\cat{V}_{lc}(n,L)$, respectively. From Corollary \ref{cor:1.1}, we are done.

\end{proof}

We will show that  $\cat{M}'_{lc}{(n,L)}$ does not admit
any algebraic function. In view of Proposition \ref{lem:3}, 
it is enough to show that $\cat{C}(\Theta)$ does not have 
any non constant algebraic function. The proof of the Theorem \ref{thm:1.3}
is very similar to the proof of the Theorem $4.3$ in \cite{BR}.

\begin{proof}[\bf Proof of Theorem \ref{thm:1.3}]
\label{proof_thm_1.3}
Let $\At{\Theta}$ be the Atiyah bundle over 
$\cat{U}_{L}(n,d)$ associated to ample line bundle $\Theta$ as described in \eqref{eq:1.2}, and
$\p(\At{\Theta})$ be the projectivization of 
$\At{\Theta}$, that is, $\p(\At{\Theta})$ parametrises hyperplanes in $\At{\Theta}$.
Let $\p({T \cat{U}_{L}})$ be the projectivization of the tangent bundle $T\cat{U}_{L}(n,d)$. 
Notice that $\p(T\cat{U}_{L})$ is a subvariety of 
$\p (\At{\Theta})$, and  $\p(T\cat{U}_{L})$ is the zero
locus of the of a section of the tautological line bundle
$\struct{\p (\At{\Theta})}(1)$. Now, observe that
$\cat{C}(\Theta) = \p(\At{\Theta}) \setminus \p(T \cat{U}_{L})$. Then we have

\begin{equation}
\label{eq:27}
\coh{0}{\cat{C}(\Theta)}{\struct{\cat{C}(\Theta)}} =
\varinjlim_{k} \coh{0}{\p \At{\Theta}}{\struct{\p \At{\Theta}}(k)} = \varinjlim_{k} \coh{0}{\cat{U}_{L}(n,d)}{\cat{S}^k\At{\Theta}}
\end{equation}

where $\cat{S}^k \At{\Theta}$ denotes the $k$-th symmetric powers of $\At{\Theta}$.
Consider the symbol operator 
\begin{equation}
\label{eq:28}
\sigma:\At{\Theta} \to   T\cat{U}_{L}(n,d)
\end{equation}
given in  \eqref{eq:1.2}.
This induces a morphism 
\begin{equation}
\label{eq:29}
\cat{S}^k(\sigma): \cat{S}^k \At{\Theta} \to 
\cat{S}^k T \cat{U}_{L}(n,d) 
\end{equation}
of $k$-th symmetric powers.
Now, because of the following composition 
\begin{equation*}
\cat{S}^{k-1} \At{\Theta} = \struct{\cat{U}_{L}(n,d)} \otimes \cat{S}^{k-1} \At{\Theta} \hookrightarrow 
\At{\Theta} \otimes \cat{S}^{k-1} \At{\Theta} \to 
\cat{S}^{k} \At{\Theta},
\end{equation*}
we have 
\begin{equation}
\label{eq:30}
\cat{S}^{k-1} \At{\Theta} \subset \cat{S}^k \At{\Theta}
~~~ \mbox{for all}~ k \geq 1.
\end{equation}

Thus, we get a short exact sequence of vector bundles over
$\cat{U}_{L}(n,d)$,
\begin{equation}
\label{eq:31}
0 \to \cat{S}^{k-1} \At{\Theta} \to \cat{S}^{k} \At{\Theta} \xrightarrow{\cat{S}^k (\sigma)} \cat{S}^k T \cat{U}_{L}(n,d) \to 0.
\end{equation}
 
 In other words, we get a filtration 
 \begin{equation}
 \label{eq:32}
 0 \subset \cat{S}^0 \At{\Theta} \subset \cat{S}^1 \At{\Theta} \subset \ldots \subset \cat{S}^{k-1} \At{\Theta} \subset \cat{S}^k \At{\Theta} \subset \ldots
 \end{equation}

 such that 
 \begin{equation}
 \label{eq:33}
 \cat{S}^{k} \At{\Theta} / \cat{S}^{k-1} \At{\Theta} \cong \cat{S}^k T \cat{U}_{L}(n,d)~~~ \mbox{for all}~ k \geq 1.
 \end{equation}
 
 Above filtration in \eqref{eq:32} gives following 
 increasing chain of $\C$-vector spaces
 
 \begin{equation}
 \label{eq:34}
 \coh{0}{\cat{U}_{L}(n,d)}{\struct{\cat{U}_{L}(n,d)}} \subset \coh{0}{\cat{U}_{L}(n,d)}{\cat{S}^1\At{\Theta}}
 \subset \ldots
 \end{equation}
 
 To prove \eqref{eq:C}, it is enough to show that 
\begin{equation}
\label{eq:35}
\coh{0}{\cat{U}_{L}(n,d)}{\cat{S}^{k-1}\At{\Theta}}
\cong \coh{0}{\cat{U}_{L}(n,d)}{\cat{S}^k\At{\Theta}}
~~~ \mbox{for all}~ k \geq 1.
\end{equation}

Since,
\begin{equation*}
\label{eq:36}
\frac{\cat{S}^k \At{\Theta}}{ \cat{S}^{k-2} \At{\Theta}}
\cong \frac{\cat{S}^k T \cat{U}_{L}(n,d)}{\cat{S}^{k-1} 
T \cat{U}_{L}(n,d)},
\end{equation*}
we have following commutative diagram
\begin{equation}
\label{eq:cd1}
\xymatrix{
0 \ar[r] & \cat{S}^{k-1} \At{\Theta} \ar[d] \ar[r] & \cat{S}
^k \At{\Theta} 
\ar[d] \ar[r]^{\cat{S}^k(\sigma)} & \cat{S}^k T \cat{U}
_{L}(n,d) \ar[d] \ar[r] & 0 \\
0 \ar[r] & \cat{S}^{k-1} T \cat{U}_{L}(n,d) \ar[r] & 
\frac{\cat{S}^k \At{\Theta}}{\cat{S}^{k-2} \At{\Theta}} \ar[r] & \cat{S}^k T \cat{U}_{L}
(n,d) \ar[r] & 0 
}
\end{equation}

which gives rise to a following commutative 
diagram of long exact sequences
\begin{equation}
\label{eq:cd2}
\xymatrix{
\cdots \ar[r] & \coh{0}{\cat{U}_{L}(n,d)}{\cat{S}^k T 
\cat{U}_{L}(n,d)} \ar[d] \ar[r]^{\delta'_k} & \coh{1}
{\cat{U}_{L}(n,d)}{\cat{S}^{k-1}\At{\Theta}} \ar[d] 
\ar[r] & \cdots  \\
\cdots \ar[r] & \coh{0}{\cat{U}_{L}(n,d)}{\cat{S}^k T 
\cat{U}_{L}(n,d)}       \ar[r]^{\delta_k} & \coh{1}
{\cat{U}_{L}(n,d)}{\cat{S}^{k-1} T \cat{U}_{L}(n,d)} 
\ar[r] & \cdots }
\end{equation}
To show \eqref{eq:35}, it is enough to prove that 
the boundary operator $\delta'_k$ is injective for all $k
\geq 1$, which is equivalent to showing that 
the boundary operator 
\begin{equation}
\delta_k: \coh{0}{\cat{U}_{L}(n,d)}{\cat{S}^k T \cat{U}
_{L}(n,d)}      \to  \coh{1}{\cat{U}_{L}(n,d)}
{\cat{S}^{k-1} T \cat{U}_{L}(n,d)}
\end{equation}
is injective for every $k \geq 1$.

Now, we will describe $\delta_k$ using the first Chern
class $c_1(\Theta) \in \coh{1}{\cat{U}_{L}(n,d)}
{ T^* \cat{U}_{L}(n,d)}$ of the ample line bundle $\Theta$ over $\cat{U}_{L}(n,d)$ .The cup product with $ kc_1(\Theta)$ gives rise to a homomorphism
\begin{equation}
\label{eq:37}
\mu: \coh{0}{\cat{U}_{L}(n,d)}{\cat{S}^k T \cat{U}
_{L}(n,d)} \to \coh{1}{\cat{U}_{L}(n,d)}{\cat{S}^k T \cat{U}
_{L}(n,d) \otimes T^* \cat{U}_{L}(n,d)}
\end{equation}
Also, we have a canonical homomorphism of vector bundles
\begin{equation*}
\label{eq:38}
\beta:\cat{S}^k T\cat{U}_{L}(n,d) \otimes T^*\cat{U}_{L}(n,d) \to  \cat{S}^{k-1}T \cat{U}_{L}(n,d)
\end{equation*}
which induces a morphism of  \C-vector spaces
\begin{equation}
\label{eq:39}
\beta^*:\coh{1}{\cat{U}_{L}(n,d)}{\cat{S}^k T \cat{U}
_{L}(n,d) \otimes T^* \cat{U}_{L}(n,d)} \to \coh{1}
{\cat{U}_{L}(n,d)}{\cat{S}^{k-1} T \cat{U}
_{L}(n,d)}.
\end{equation}
So, we get a morphism
\begin{equation}
\label{eq:40}
\tilde{\mu} = \beta^* \circ \mu: \coh{0}{\cat{U}_{L}(n,d)}{\cat{S}^k T 
\cat{U}_{L}(n,d)} \to \coh{1}{\cat{U}_{L}(n,d)}
{\cat{S}^{k-1} T \cat{U}_{L}(n,d)}.
\end{equation}
Then $\tilde{\mu} = \delta_k$. 
It is sufficient to show that $\tilde{\mu}$ is injective.

Moreover, we have natural projection
\begin{equation}
\label{eq:41}
\eta: T^* \cat{U}_{L}(n,d) \to \cat{U}_{L}(n,d)
\end{equation}
and 
\begin{equation}
\label{eq:42}
\eta_* \eta^* \struct{\cat{U}_{L}(n,d)} = \oplus_{k \geq 0} \cat{S}^k T \cat{U}_{L}(n,d).
\end{equation}
Thus, we have 
\begin{equation}
\label{eq:43}
\coh{j}{T^*\cat{U}_{L}(n,d)}
{\struct{ T^* \cat{U}_{L}(n,d)}} = \oplus_{k \geq 0}
\coh{j}{\cat{U}_{L}(n,d)}
{\cat{S}^{k} T \cat{U}_{L}(n,d)} ~~~ \mbox{for all}~ j \geq 0.
\end{equation}

Now, we use Hitchin fibration to compute $\coh{j}{T^*\cat{U}_{L}(n,d)}{\struct{ T^* \cat{U}_{L}(n,d)}}$.
Let 
\begin{equation}
\label{eq:44}
h: T^* \cat{U}_{L}(n,d) \to B_n = \oplus_{i =2}^{n}\coh{0}{X}
{K_{X}^{i}}
\end{equation}
be the Hitchin map defined by sending a pair $(E, \phi)$
to $ \sum_{i =1}^{n} trace(\phi^i)$. Notice that the base
of the Hithcin map $h$ in \eqref{eq:44} is a vector space
over $\C$ of dimension $n^2(g-1)+1$.

Let $b \in B_n$. Then $h^{-1}(b) = A \setminus F$, where
$A$ is some abelian variety and $F$ is a subvariety of 
$A$ with $\mbox{codim}(F,A) \geq 3$ (for more details see \cite{BNR}, \cite{H}), and we will be using this fact 
showing that $\tilde{\mu}$ is injective.

Let $g: T^*\cat{U}_{L}(n,d) \to \C$ be an algebaric function. Then its restriction  $g|_{h^{-1}(b)}: h^{-1}(b) \to \C$ to $h^{-1}(b)$ for every 
$b \in B_n$ is an algebraic function. Since $\mbox{codim(A,F)} \geq 3$, $g|_h^{-1}(b)$ extended to 
a unique algebraic function $\tilde{g}:A \to \C$.
$A$ being an abelian variety, $\tilde{g}$ is a constant 
function. Thus, on each fibre $h^{-1}(b)$, $g$ is constant, and hence gives an algebraic function on
$B_n$.  

Set $\cat{B} = \mbox{d}( \coh{0}{B_n}
{\struct{B_n}}) \subset  \coh{0}{B_n}
{\Omega^1_{B_n}}$ the space of all exact algebraic $1$-form.
Define a map 
\begin{equation}
\label{eq:45}
\theta: \coh{0}{T^*\cat{U}_{L}(n,d)}
{\struct{ T^* \cat{U}_{L}(n,d)}} \to \cat{B}
\end{equation}
by $t \mapsto dg$, where $g$ is the function which is defined by descent of $t$.  Then $\theta$ is an isomorphism.

From \eqref{eq:43} and \eqref{eq:45}, we have
\begin{equation}
\label{eq:46}
\theta:\oplus_{k \geq 0}
\coh{0}{\cat{U}_{L}(n,d)}
{\cat{S}^{k} T \cat{U}_{L}(n,d)} \to \cat{B}
\end{equation}
which is an isomorphism.

Let $T_h = T_{T^* \cat{U}_{L}(n,d) / B_n} = \SKer{dh}$
be the relative tangent sheaf on $T^*\cat{U}_{L}(n,d)$,
where $dh: T(T^*\cat{U}_{L}(n,d)) \to h^*TB_n$ morphism 
of bundles.

 Note that 
$ \coh{0}{B_n}{\Omega^1_{B_n}} \subset  \coh{0}{T^*\cat{U}_{L}(n,d)}
{T_h} $, and hence from \eqref{eq:46}, we have
an injective homomorphism
\begin{equation}
\label{eq:47}
\nu: \cat{B} = \oplus_{k \geq 0} \theta(
\coh{0}{\cat{U}_{L}(n,d)}
{\cat{S}^{k} T \cat{U}_{L}(n,d)}) \to \coh{0}{T^*\cat{U}_{L}(n,d)}{T_h}.
\end{equation}

Consider the morphism 

$\coh{0}{T^*\cat{U}_{L}(n,d)}{T_h} \to \coh{1}{T^*\cat{U}_{L}(n,d)}{T_h \otimes T^* T^* \cat{U}_{L}(n,d)}$ defined 
by taking cup product with the first Chern class $c_1(\eta^* \Theta) \in  \coh{1}{T^*\cat{U}_{L}(n,d)}{ T^* T^* \cat{U}_{L}(n,d)}$.

Using the pairing $T_h \otimes T^* T^* \cat{U}_{L}(n,d) \to \struct{T^* \cat{U}_{L}(n,d)}$, we get a homomorphism
\begin{equation}
\label{eq:48}
\psi: \coh{0}{T^*\cat{U}_{L}(n,d)}{T_h} \to \coh{1}{T^*\cat{U}_{L}(n,d)}{\struct{T^*\cat{U}_{L}(n,d)}}
\end{equation}

Since $c_1(\eta^* \Theta) = \eta^*(c_1 \Theta)$, we have
\begin{equation}
\label{eq:49}
k \psi \circ \nu \circ \theta(\omega_k) = \tilde{\mu}(\omega_k),
\end{equation}
for all $\omega_k \in \coh{0}{\cat{U}_{L}(n,d)}
{\cat{S}^{k} T \cat{U}_{L}(n,d)})$.
Since $\nu$ and $\theta$ are injective homomorphisms, it is enough to 
show that $\psi|_{\nu(\cat{B})}$ is injective homomorphism. 
Let $\omega \in \cat{B} \setminus \{0\}$ be a non-zero
exact $1$-form. Choose $b \in B_n$ such that $\omega(b) \neq 0$. As previously discussed $h^{-1}(b) = A \setminus F$, where $A$ is an abelian variety and $F$ is a subvariety of $A$ such that $\mbox{codim}(F,A) \geq 3$.
Now, $\psi(\nu(\omega)) \in \coh{1}{T^*\cat{U}_{L}(n,d)}{\struct{T^*\cat{U}_{L}(n,d)}}$ and we have restriction map $\coh{1}{T^*\cat{U}_{L}(n,d)}{\struct{T^*\cat{U}_{L}(n,d)}} \to \coh{1}{h^{-1}(b)}{\struct{h^{-1}(b)}}$.
Since $ \omega(b) \neq 0$, $\psi(\nu(\omega)) \in \coh{1}{h^{-1}(b)}{\struct{h^{-1}(b)}}$.
Because of the following isomorphisms
\begin{equation*}
\label{eq:50}
 \coh{1}{h^{-1}(b)}{\struct{h^{-1}(b)}} \cong  \coh{1}{A}{\struct{A}} \cong  \coh{0}{A}{TA},
\end{equation*}
it follows that $\psi(\nu(\omega)) \neq 0$.
This completes the proof.

\end{proof}

Since $\cat{M}'_{lc}(n,L)$ is a open dense subset of 
$\cat{M}_{lc}(n,L)$, we have following
\begin{corollary}
\label{cor:2}
$\coh{0}{\cat{M}_{lc}(n,L)}{\struct{\cat{M}_{lc}(n,L)}} =\C.$
\end{corollary}

Now, for the pair $(L_0,D_{L_0})$ where $L_0 = \otimes_{i = 1}^{m}
\struct{X}(-n \lambda_i x_i) $ and $D_{L_0}$ the 
logarithmic connection defined by the de Rham 
differential as described in section \eqref{Mod-log-conn}, consider the moduli space $\cat{M}_{lc}(n,L_0)$.
We show that the moduli space $\cat{M}_{lc}(n,L_0)$ 
admits non-constant holomorphic functions. Consider
the Betti moduli space $\cat{B}_g$ described in section 
\eqref{Mod-log-conn}, which is an affine variety.

Let $\gamma_j \in \pi_1(X_0,x_0)$. Define a function $f_{jk}: \cat{B}_g \to \C$ by $\rho 
\mapsto \text{trace}(\rho(\gamma_j)^k)$ for $k \in \N$.
Then $f_{jk}$ are non-constant algebraic functions on $
\cat{B}_g$ for $j= 1, \ldots, m $ and $k \in \N$.
Thus $\cat{M}_{lc}(n,L_0)$ is not isomorphic to $B_g$ as 
algebraic varieties.

 Since $\cat{M}_{lc}(n,L_0)$ is biholomorphic 
to $\cat{B}_g$, $f_{jk} \circ \Phi: \cat{M}_{lc}(n,L_0) 
\to \C$ are non-constant holomorphic  functions for all
$j =1, \ldots, m$ and $k \in \N$.

\section{The Moduli space of logarithmic connection with 
arbitrary residues}
\label{Mod-log-conn-arb}
Let $X$ be a compact Riemann surface of genus($g$) $\geq 
3$ and $S = \{x_1, \ldots, x_m\}$ be a subset of distinct points  of $X$ as in section \eqref{Pre}.
By a pair $(E,D)$ over $X$, we mean that 
\begin{enumerate}
\item $E$ is a holomorphic vector bundle over $X$ of degree $d$ and rank $n$.
\item $n$ and $d$ are mutually coprime.
\item $D$ is a logarithmic connection in $E$ singular 
over $S$. 

\end{enumerate}
We call such a pair $(E,D)$ logarithmic connection 
on $X$ singular over $S$.

Now, given such a pair $(E,D)$,  from \cite{O}, Theorem 3, we have
\begin{equation}
\label{eq:51}
d + \sum_{j=1}^m \tr{Res(D,x_j)} = 0,
\end{equation}
where  $Res(D,x_j) \in \ENd{E(x_j)}$, for all $j =1, \ldots, m$.

Let $\cat{N}_{lc}(n,d)$ be the moduli space which 
parametrises isomorphism class of pairs $(E,D)$.
Then $\cat{N}_{lc}(n,d)$ is a separated quasi-projective
scheme over $\C$ (see \cite{N}).
Let $\cat{N}'_{lc}(n,d)$ be a subset of $\cat{N}_{lc}
(n,d)$, whose underlying vector bundle is stable.
Let $(E,D)$ and $(E,D')$ be two points in $\cat{N}'_{lc}(n,d)$. Then 
\begin{equation}
\label{eq:52}
D-D' \in \coh{0}{X}{\ENd{E} \otimes \Omega^1_X(\log S)}.
\end{equation}
Next, for $\theta \in \coh{0}{X}{\ENd{E} \otimes \Omega^1_X(\log S)}$, we have $(E, D+ \theta) \in \cat{N}'_{lc}(n,d)$.
Notice the difference between the affine spaces when residue is fixed and otherwise.
Thus, the space of all logarithmic connections $D$ on a
given stable vector bundle $E$ singular over $S$, is an
affine space modelled over $\coh{0}{X}{\ENd{E} \otimes \Omega^1_X(\log S)}$.
Let 
\begin{equation}
\label{eq:53}
q: \cat{N}'_{lc}(n,d) \to \cat{U}(n,d)
\end{equation}
be the natural projection defined by sending $(E,D)$ to 
$E$. 
Given $E \in \cat{U}(n,d)$. Choose a set of complex numbers $\alpha_1,\ldots, \alpha_m$ which satisfies the
following equation
\begin{equation}
\label{eq:53.1}
d+ n \sum_{j =1}^m \alpha_j = 0.
\end{equation}
Since $E$ is stable, from \cite{B}, Proposition $1.2$,
$E$ admits a logarithmic connection $D$ singular over 
$S$. Thus, $q$ is a surjective map,
 and dimension of each
fibre $q^{-1}(E)$ is $n^2(g-1+m)$.
We have following result very similar to the
Theorem \ref{thm:1.1}.

\begin{theorem}
\label{thm:4}
There exists an algebraic vector bundle $\tilde{\pi}: \tilde{ \Xi} \to \cat{U}(n,d)$ such that 
$\cat{N}'_{lc}(n,d)$ is embedded in $\p (\tilde{\Xi})$ with $\p (\tilde{\Xi}) \setminus \cat{N}'_{lc}(n,d)$ as the
hyperplane at infinity.
\end{theorem} 
\begin{proof}
Proof is very similar to the proof of Theorem \ref{thm:1.1}. 
\end{proof}

Next, the morphism $q$ defined in \eqref{eq:53} induces 
a homomorphism 
\begin{equation}
\label{eq:54}
q^*: \Pic{\cat{U}(n,d)} \to  \Pic{\cat{N}'_{lc}(n,d)}
\end{equation}
of Picard groups, that sends line bundle $\eta$ over 
$\cat{U}(n,d)$ to a line bundle $q^*\eta$ over $
\cat{N}'_{lc}(n,d)$ as described in subsection 
\eqref{Picard}.  Again, we record a result similar 
to the Theorem \ref{thm:1.2}.

\begin{theorem}
\label{thm:5}
The homomorphism $q^*: \Pic{\cat{U}(n,d)} \to \Pic{\cat{N}'_{lc}(n,d)}$ is an isomorphism of groups.
\end{theorem}
\begin{proof} Proof is very similar to the proof of Theorem \ref{thm:1.2}.
\end{proof}

Now, fix a pair $(L,D_L)$, where $L$ is a holomorphic 
vector bundle of degree $d$ and $D_L$ is a fixed 
logarithmic connections on $L$ singular over $S$.
Let $\cat{N}_{lc}(n,L)$ denote the moduli space 
parametrising all pairs $(E,D)$ such that 
\begin{enumerate}
\item $E$ is a holomorphic vector bundle 
over $X$ of rank $n$ and degree $d$ with
 $\bigwedge^{n}E \cong L $, and $n$ and $d$ are mutually 
 coprime.
\item $D$ is a logarithmic connection in $E$ singular 
over $S$ with $Res(D,x_j)$ $\in$  $Z(\mathfrak{gl}(n,\C))$, and $\tr{Res(D,x_j)} \in \Z$,
where $Z(\mathfrak{gl}(n,\C))$ denotes the centre of 
$\mathfrak{gl}(n,\C)$.
\item the logarithmic connection on $\bigwedge^{n}E$ induced by
$D$ coincides with the given logarithmic connection $D_{L}$ on $L$.
\end{enumerate}
Then, Lemma \eqref{lem:1} holds for such a pair $(E,D)$, and by Proposition \ref{pro:2}, $(E,D)$ is irreducible.

Let $\cat{N}'_{lc}(n,L)$ be the subset of $\cat{N}_{lc}
(n,L)$ whose underlying vector bundle is stable.
Let 
\begin{equation}
\label{eq:55}
q_0: \cat{N}'_{lc}(n,L) \to \cat{U}_L(n,d)
\end{equation}
be the natural projection sending $(E,D)$ to $E$. 
Then, we have following results similar to Proposition \ref{prop:7.1} and Proposition \ref{prop:8}. 
\begin{proposition}
\label{prop:9}
There exists an algebraic vector bundle $\pi:\tilde{ \Xi}' \to \cat{U}_L(n,d)$ such that 
$\cat{N}'_{lc}(n,L)$ is embedded in $\p (\tilde{\Xi}')$ with $\p (\tilde{\Xi}') \setminus \cat{N}'_{lc}(n,L)$ as the
hyperplane at infinity.
\end{proposition}

\begin{proposition}
\label{prop:10}
The homomorphism $q_0^*: \Pic{\cat{U}_{L}(n,d)} \to 
\Pic{\cat{N}'_{lc}(n,L)}$ defined by $\xi \mapsto q_0^* 
\xi$ is an isomorphism of groups.
\end{proposition}

Note that $q_0: \cat{N}'_{lc}(n,L) \to \cat{U}_L(n,d) $
is not a $\Omega^1_{\cat{U}_L(n,d)}$-torsor, and therefore we cannot apply the same technique as in previous
section\eqref{fun} to compute the algebraic functions 
on $\cat{N}'_{lc}(n,L)$.

Next, let 
\begin{equation*}
\label{eq:55.1}
V = \{(\alpha_1, \ldots, \alpha_m) \in \C^m|~~ n \alpha_j \in \Z~~ \text{and}~~ d+ n \sum_{j=1}^{m} \alpha_j = 0 \}
\end{equation*}

Define a map 
\begin{equation}
\label{eq:56}
\Phi: \cat{N}'_{lc}(n,L) \to V
\end{equation}
by $(E,D) \mapsto (\tr{Res(D,x_1)}/n, \ldots, \tr{Res(D,x_m)}/n)$.

\begin{proof}[\bf Proof of Theorem \ref{thm:1.4}]
\label{proof_thm_1.4}
Let $(\alpha_1, \ldots, \alpha_m) \in V$. Then 
$\Phi^{-1}((\alpha_1, \ldots, \alpha_m))$ is the moduli
space of logarithmic connections with fixed residues 
$\alpha_j \id{E(x_j)}$, which is isomorphic to 
$\cat{M}'_{lc}(n,L)$ follows from Corollary \ref{cor:1.1}. Let $g: \cat{N}'_{lc}(n,d) \to \C$ be an algebraic 
function. Then $g$ restricted to each fibre of $\Phi$
is an algebraic function on the moduli space isomorphic
to $\cat{M}'_{lc}(n,L)$. Now, from Theorem \ref{thm:1.3},
$g$ is constant on each fibre and thus defining a 
function from $V \to \C$. This completes the proof.
\end{proof}

Similarly, we define a map
\begin{equation}
\label{eq:57}
\Psi: \cat{N}_{lc}(n,L) \to V
\end{equation}
by $(E,D) \mapsto (\tr{Res(D,x_1)}/n, \ldots, 
\tr{Res(D,x_m)}/n)$.
We have following
\begin{theorem}
\label{thm:7}
Any algebraic function on $\cat{N}_{lc}(n,L)$ factor 
through the surjective map $\Psi: \cat{N}_{lc}(n,L) \to V$ as defined in \eqref{eq:57}.
\end{theorem}
\begin{proof}
Let $g: \cat{N}_{lc}(n,L) \to \C$ be an algebraic function. Then restriction of $g$ to each fibre of 
$\Psi$ is a constant function, follows from Corollary
\ref{cor:2}, and hence defining a function from $V \to
\C$.
\end{proof}

\section*{Acknowledgements}  
The author would like to thank referees for their detailed 
and helpful comments.  
The author is deeply grateful to Prof. Indranil Biswas
for suggesting the problem, and helpful discussions, and 
 would like  to thank  his Ph.D. advisor Prof. 
N. Raghavendra for numerous discussions and his 
guidance.

\end{document}